\documentclass[]{amsart}
\usepackage{amsthm}
\usepackage{amsfonts}
\usepackage{amssymb}
\usepackage{amsmath}
\usepackage{verbatim}
\usepackage{mathtools}
\usepackage{setspace}
\usepackage{array}
\usepackage{url}
\newcolumntype{L}{>{$}l<{$}} % math-mode version of "l" column type
\newcolumntype{C}{>{$}c<{$}} % math-mode version of "l" column type
\newcolumntype{R}{>{$}r<{$}} % math-mode version of "l" column type
\newtheorem{thm}{Theorem}[section]
\newtheorem{prop}[thm]{Proposition}
\newtheorem{lemma}[thm]{Lemma}
\newtheorem{cor}[thm]{Corollary}
\newtheorem{conj}[thm]{Conjecture}

\theoremstyle{definition}
\newtheorem{defin}[thm]{Definition}

\newtheorem{rmk}[thm]{Remark}
\newcommand{\normtext}[1]{\text{\normalfont{#1}}}

\newcommand{\oddO}{H}
\newcommand{\expDim}[2]{\normtext{expDim}(#1,#2)}
\newcommand{\econe}[2]{\overline{\Gamma}_{#1}(#2)}
\title{Maximal rank subgroups and strong functoriality of the additive eigencone}
\author{Michael Schuster}
\date{}
\begin{document}
\begin{abstract}
Let $G$ be a simple connected complex Lie group.  The additive eigencone $\econe{n}{G}$ is a polyhedral cone
containing the set of solutions to the additive eigenvalue problem, a generalization of the Hermitian
eigenvalue problem.  The additive eigencone is functorial, and for certain subgroups satisfies a
stronger functoriality property: the eigencone of the subgroup is determined by the inequalities
of the larger eigencone.  Belkale and Kumar first studied this property for subgroups
invariant under a diagram automorphism of $G$.  We study a new class of subgroups 
arising from centralizers of torus elements that have
the strong eigencone functoriality property.
\end{abstract}
%%%%%%%%%%%%%%  INTRODUCTION  %%%%%%%%%%%%%%
\maketitle
\section{Introduction}\label{IntroSec}
For any Hermitian matrix $A$, let $\lambda$ denote its set of eigenvalues.  Given 
two sets of eigenvalues $\lambda$ and $\mu$, the Hermitian eigenvalue problem
asks: For which sets of eigenvalues $\nu$ do there exist Hermitian matrices $A$, $B$,
and $C$ such that $A+B = C$?  The solution is that the set of tuples of eigenvalues
$(\lambda, \mu, \nu)$ satisfying the Hermitian eigenvalue problem forms a convex
polyhedral cone, whose facets are parametrized by certain cohomology products.
This problem has a long history, starting with the work of Weyl \cite{WEYL}, which led to
Horn's conjectures about inequalities determining the facets of this cone \cite{HRN62},
which were proven by the combined work of Klyachko \cite{KLY98} and Knutson-Tao \cite{KNTAO}.
For more on the history of this problem and the methods used to solve it,
see Fulton's survey \cite{FEIG}.

The Hermitian eigenvalue problem can be generalized both to an arbitrary number of matrices
and to an arbitrary connected semisimple complex algebraic group $G$ (for a survey on this generalized problem
see \cite{AEPS}).  The set of solutions
of this problem again forms a convex polyhedral cone $\econe{n}{G}$ called the 
\emph{additive eigencone}, and contains the saturated tensor semigroup as a subset.  One of the properties of the additive eigencone is that it is functorial: 
given a group homomorphism $M \rightarrow G$ which maps the maximal compact subgroup
of $M$ into the maximal compact subgroup of $G$, we get a piecewise-linear map $\econe{n}{M} \rightarrow \econe{n}{G}$ \cite{KLMJ09}.
In \cite{BKISO} Belkale and Kumar studied subgroups of $\text{SL}(r+1)$ that exhibited a stronger
functoriality property.  An eigencone $\econe{n}{G}$ is a subcone of the cone of tuples
dominant coweights $\mathfrak{h}_{+,G}^n$, and for a subgroup $M \subseteq G$
the map of eigencones $\econe{n}{M} \rightarrow \econe{n}{G}$ is the restriction of a map
$\phi:\mathfrak{h}_{+,M}^n \rightarrow \mathfrak{h}_{+,G}^n$.  The strong functoriality
property is the condition that $\vec{\lambda} \in \econe{n}{M}$ if and only if $\phi(\vec{\lambda}) \in \econe{n}{G}$.
Belkale and Kumar proved that $\text{SO}(2r+1) \subseteq \text{SL}(2r+1)$ and 
$\text{Sp}(2r) \subseteq \text{SL}(2r)$ have the strong eigencone functoriality property,
and conjectured that any subgroup that is the fixed subgroup of a diagram automorphism
of $G$ will also have this property.  This conjecture was resolved by the work of Braley and Lee \cite{BRTHS,LEETHS}.

In this article we study the strong functoriality property for groups arising from inner automorphisms.
We say that a subgroup $M \subseteq G$ \emph{induces a sub-eigencone} if
\begin{itemize}
\item The induced map $\mathfrak{h}_{+,M} \rightarrow \mathfrak{h}_{+,G}$ is an isometric embedding and,
\item The map of eigencones $\econe{n}{M} \rightarrow \econe{n}{G}$ has the strong functoriality property.
\end{itemize}
Then our main theorem is the following.

\begin{thm} \label{MainThm}
For any $n \geq 3$ the following subgroups induce sub-eigencones.  (We describe the subgroups using the simple roots 
$\alpha_1, \ldots, \alpha_r$ of the larger group.)
\begin{enumerate}
\item The subgroup $\normtext{SO}(2r-1) \subseteq \normtext{SO}(2r+1)$ fixed by the diagram automorphism of the subgroup $\normtext{SO}(2r)$
corresponding to the type $\normtext{D}_r$ sub-root-system of $\normtext{SO}(2r+1)$.
\item The subgroup $\normtext{Sp}(2r-2) \subseteq \normtext{Sp}(2r)$ corresponding to the type $\normtext{C}_{r-1}$ 
sub-root-system with simple roots $\beta_1 = \alpha_1, \ldots,\beta_{r-2} = \alpha_{r-2}, \beta_{r-1} = 2\alpha_{r-1}+\alpha_r$.
\item The subgroup $\normtext{SO}(2r-3) \subseteq \normtext{SO}(2r)$ contained (as above) in the subgroup $\normtext{SO}(2r-1)$
fixed by the diagram automorphism of $\normtext{SO}(2r)$.
\item The subgroup $\normtext{SL}(2) \subseteq \normtext{G}_2$ corresponding to the sub-root-system with
simple root $\beta = 3\alpha_1 + 2\alpha_2$.
\item The subgroup $\normtext{G}_2 \subseteq \normtext{F}_4$ obtained as follows: the group $\normtext{F}_4$ contains
a subgroup of type $\normtext{B}_4$ corresponding to the sub-root-system with simple roots 
$\beta_1 = \alpha_2+2\alpha_3+2\alpha_4, \beta_2 = \alpha_1,  \beta_3 = \alpha_2, \beta_4 = \alpha_3$; the 
$\normtext{B}_4$ subgroup contains a type $\normtext{D}_4$ subgroup as above; 
finally the $\normtext{G}_2$ subgroup is the subgroup fixed by the diagram automorphism of $\normtext{D}_4$.
\end{enumerate}
\end{thm}
\begin{rmk}
The property of inducing a sub-eigencone is transitive, so that, for example, the above theorem implies
that $\econe{n}{\normtext{SO}(2r+1)}$ contains sub-eigencones of type $\text{B}_s$ for any $s<r$.
\end{rmk}

The images of these sub-eigencones are also interesting.  The normalized Killing form defines an isomorphism
between $\mathfrak{h}$ and $\mathfrak{h}^*$.  Therefore we can describe subcones of an eigencone in
terms of the fundamental weights $\omega_1, \ldots, \omega_r \in \mathfrak{h}^*$ of $G$.  In each case of the above
theorem, the images of the sub-eigencones can be described as subcones of $\econe{n}{G}$
obtained by setting the coefficients of fundamental weights to zero.

\begin{cor}\label{MainCor}
For any $n \geq 3$ we have the following. In each case we denote by $\nu_i$ the fundamental weights of the
smaller group $M$.
\begin{enumerate}
\item For any $1 \leq s < r$, the subcone of $\econe{n}{\normtext{Sp}(2r)}$ in which the coefficients of $\omega_{s+1}, \ldots, \omega_r$ are zero
for every weight $\lambda_i$ in $\vec{\lambda}$ is isomorphic to $\econe{n}{\normtext{Sp}(2s)}$.  The weights satisfy
$(\sum_{i=1}^s a_i \omega_i)_{|M} = \sum_{i=1}^s a_i \nu_i$.  The same is true for 
$\econe{n}{\normtext{SO}(2s+1)} \subseteq \econe{n}{\normtext{SO}(2r+1)}$, but in this case 
the weights satisfy
$(\sum_{i=1}^s a_i \omega_i)_{|M} = \sum_{i=1}^{s-1} a_i \nu_i + 2a_s \nu_s$.
\item The subcone of $\econe{n}{\normtext{SO}(2r)}$ in which the coefficients of 
$\omega_{r-1}$, and $\omega_r$ are zero
for every weight is isomorphic to $\econe{n}{\normtext{SO}(2r-3)}$.  The weights satisfy
($\sum_{i=1}^{r-2} a_i \omega_i)_{|M} = \sum_{i=1}^{r-3} a_i \nu_i + 2a_{r-2}\nu_{r-2}$.
\item The subcone of $\econe{n}{\normtext{G}_2}$ in which the coefficients of $\omega_1$ are zero
for every weight is isomorphic to $\econe{n}{\normtext{SL}(2)}$.  The weights satisfy
$ (a \omega_2)_{|M} =  2a \nu$.
\item The subcone of $\econe{n}{F_4}$ in which the coefficients of $\omega_3$ and $\omega_4$ are zero
for every weight is isomorphic to $\econe{n}{\normtext{G}_2}$.  The weights satisfy
$ (a \omega_1 + b \omega_2)_{|M} =  3b \nu_1 + a \nu_2$.
\end{enumerate}
\end{cor}

In types B and C, we additionally have a projection from the larger eigencone to the smaller eigencone described
above.

\begin{thm}\label{ProjThm}
For any $1 \leq s < r$, there exists a natural projection $\pi: \econe{n}{\normtext{Sp}(2r)} \rightarrow \econe{n}{\normtext{Sp}(2s)}$
such that $\pi$ is surjective and the inclusion map $\iota: \econe{n}{\normtext{Sp}(2s)} \rightarrow \econe{n}{\normtext{Sp}(2r)}$ 
is a section of $\pi$.  The map $\pi$ is given by
\[
\sum_{i=1}^r a_i \omega_i \mapsto \sum_{i=1}^{s-1} a_i \nu_i + \left(\sum_{i=s}^r a_i \right) \nu_s.
\]
The same holds for $\normtext{SO}(2r+1)$.
\end{thm}

\subsection{Methods}

The proof of the above theorems relies upon the determination of $\econe{n}{G}$ by inequalities
parametrized by cohomology products.  The inequalities corresponding to the regular faces of $\econe{n}{G}$
are parametrized by products $\sigma_{w_1} \cdots \sigma_{w_n} = 1[pt]$ such that the product
is additionally \emph{Levi-movable}.  To prove a subgroup $M \subseteq G$ induces a sub-eigencone,
we need to relate cohomology products over homogeneous spaces of $M$ with products over
homogeneous spaces of $G$.  The relation is induced by the inclusion of Weyl groups $W_M \subseteq W_G$,
which gives a correspondence between Schubert cells of $M/Q$ and Schubert cells of $G/P$.
(See Section \ref{EigenconeSec} for more details on the inequalities of the eigencone and 
the strategy of the proof of the above theorems.)

The following theorem is our main geometric result. Together with the work in \cite{BRTHS,LEETHS}, this theorem
is sufficient to prove Theorem \ref{MainThm}.  Here $G=\text{Sp}(2r)$ and $M=\text{Sp}(2(r-1))$ is
the subgroup described in Theorem \ref{MainThm}, and $\text{IG}(k,2r)$ denotes the Grassmannian
of $k$-dimensional isotropic subspaces.

\begin{thm}\label{CohomThm}
For classes $\sigma_1^M, \ldots, \sigma_n^M \in \normtext{H}^*(\normtext{IG}(k,2(r-1)))$, if
the product $\sigma_1^M \cdots \sigma_n^M = 1\normtext{[pt]}$ is Levi-movable, then 
the corresponding product $\sigma_1 \cdots \sigma_n$ in $\normtext{H}^*(\normtext{IG}(k,2r))$ is non-zero,
Levi-movable, and equal to a multiple of the class of a point in $G/P$.
\end{thm}

\begin{rmk}
Both the assumption that the product is equal to $1\text{[pt]}$ and that the product is Levi-movable
are necessary, and therefore the eigencone theorems
depend on the work of Belkale and Kumar, who showed that the inequalities parametrized by
these products are sufficient to define the eigencone \cite{BK06}.  In fact, the theorem
also depends on the work of Ressayre, who showed that this reduced set of inequalities are
irredundant \cite{RSS10}.
\end{rmk}

The basic strategy of the proof of this theorem is to show that the intersection of the Schubert
cells corresponding to each $\sigma_i$ can be made proper by the action of $M$ on $G/P$.
By replacing $M$ with the maximal rank subgroup $\text{Sp}(2(r-1)) \times \text{SL}(2)$ containing
it, we obtain a subgroup that acts on a finite set of orbits, making it
possible to check properness orbit-by-orbit.
This is the strategy used by Belkale-Kumar and Braley
to prove similar cohomology results \cite{BKISO,BRTHS} (they did not need to enlarge $M$ however).  
In our case it is not quite true that even this larger $M$ can make the intersection proper.  
However by first shifting one of the cells by a particular element of $G$ we can show that
general elements of $M$ will make the intersection proper.  For more details see section \ref{ProperSec}.

\subsection{Finding sub-eigencones and a general conjecture}

Finally, we want to indicate how these results were found, and make a general conjecture.
Let $G$ be a simple, simply connected algebraic
group over $\mathbb{C}$, and suppose $M \subseteq G$ is a semisimple subgroup,
such that $M$ is the fixed subgroup of an automorphism of $G$.  Assume there
is a simple factor of $M$, say $M_1$, whose highest root coincides with the highest root of  $G$.  The philosophy
behind the above results is that there should be a close relationship between the eigencone of $G$
and the eigencone of $M_1$.

The map $\econe{n}{M_1} \rightarrow \econe{n}{G}$ is in general only piecewise-linear, and may involve 
\emph{folding} $\econe{n}{M_1}$ into $\econe{n}{G}$, identifying multiple points of $\econe{n}{M}$
with a point in $\econe{n}{G}$.
This occurs because the dominant chamber of $M_1$ will not always map into the dominant chamber
of $G$.  A subgroup $M_1 \subseteq G$ that induces this folding behavior will in general not
have the strong eigencone functoriality property.  However, along the ``folds" of this map,
where there is no identification, we observed that the the strong eigencone functoriality property
does hold.  More precisely, the folds of the map $\phi: \econe{n}{M_1} \rightarrow \econe{n}{G}$ is
the set $\{ \vec{\lambda} \in \econe{n}{M_1} \mid \phi^{-1}(\phi(\vec{\lambda})) = \{\vec{\lambda}\} \}$,
and we have the following conjecture.

\begin{conj}
Along the folds of the map $\econe{n}{M_1} \rightarrow \econe{n}{G}$, we have
that $\vec{\lambda} \in \econe{n}{M_1}$ if and only if $\vec{\lambda} \in \econe{n}{G}$.
\end{conj}

The results in Theorem \ref{MainThm} arose from following this script in types B, C, D, $\text{G}_2$ and $\text{F}_4$
for maximal rank semisimple subgroups of $G$; these subgroups are all centralizers of torus elements.  Folding
does happen in types B, D and $\text{F}_4$, and the sub-eigencones listed for these groups are a combination
of the result in type C and the results of Braley and Lee.
In type A the required subgroups do not exist, since in this case any proper centralizer of a torus element is not semisimple.
We did not attempt to study types $\text{E}_6$, $\text{E}_7$, $\text{E}_8$, and the conjecture is open for 
fixed subgroups of inner automorphisms of these groups.

\subsection{Outline}

The paper is organized as follows.
\begin{itemize}
\item Section \ref{EigenconeSec}: we review the definition of the additive eigencone, and 
the determination of its facets by inequalities parametrized by cohomology products.  We also
discuss the isomorphism between the eigencones in type B and C, which is important in the next
section.
\item Section \ref{PrelimSection}: we begin the proof of Theorem \ref{CohomThm} by showing
the intersection in $\text{IG}(k,2r)$ has expected dimension zero.  
\item Section \ref{ProperSec}: we finish the proof of Theorem \ref{CohomThm} as outlined in
the introduction.  We then use this theorem to finish the proof of Theorem \ref{MainThm}.
\end{itemize}

\subsection{Notation}
Let $G$ be a simple, connected, complex algebraic group of rank $r$.  Fix a Borel subgroup $B$ and a maximal
torus $T \subseteq B$.  Let $W=W_G$ be the Weyl group of $G$.  We denote the Lie algebras of Lie groups
using fraktur script: for example, the Lie algebra of $G$ is written $\mathfrak{g}$.  Let $\mathfrak{h}$
be the Cartan algebra corresponding to the choice of torus, and let $R \subseteq \mathfrak{h}^*$
be the set of roots of $\mathfrak{g}$.  Let $R^+$ be the set of positive roots with respect to $B$,
and let $\alpha_1, \ldots, \alpha_r$ be the simple roots, ordered as in Bourbaki \cite{BLIE46}.
Let $x_1, \ldots, x_r \in \mathfrak{h}$ be the dual basis of $\alpha_1, \ldots, \alpha_r \in \mathfrak{h}^*$.
We denote the Killing form using angle brackets $\langle, \rangle$, which is normalized so that
$\langle \theta, \theta \rangle = 2$, where $\theta$ is the highest root of $G$.  The fundamental
weights $\omega_1, \ldots, \omega_r \in \mathfrak{h}^*$ are defined so that 
$\frac{2\langle \omega_i, \alpha_j \rangle}{\langle \alpha_j, \alpha_j \rangle} = \delta_{ij}$,
and we denote the dominant Weyl chamber $\mathfrak{h}_+ \subseteq \mathfrak{h}$, and its dual space
of dominant weights $\Lambda_+ \subseteq \mathfrak{h}^*$.

For any root $\beta$ we denote the corresponding reflection $s_\beta \in W$.  For a parabolic subgroup $P \supseteq B$,
let $W_P$ be the corresponding subgroup of $W$, and let $W^P$ be the set of minimal length coset representatives
of $W/W_P$.  Each $w \in W^P$ corresponds to a unique Schubert cell $C_w \subseteq G/P$.  Let $X_w$
be the closed Schubert cell, and $\sigma_w \in \text{H}^*(G/P)$ be the Poincar\'{e} dual of the homology
class of $X_w$.

\subsection{Acknowledgements}
I gladly thank Prakash Belkale and Shrawan Kumar for helpful conversations and encouragement during the preparation of
this article.

\section{The additive eigencone in types B and C}\label{EigenconeSec}

In this section we review the definition and geometry of the additive eigencone of $G$,
describe the isomorphism between the eigencones in types B and C, and explain the implications
this has for the cohomology products parameterizing the two eigencones.

\subsection{The additive eigencone of $G$}

We want to define more precisely the additive eigencone of $G$, and how its faces are described by cohomology products.
Let us first consider the algebra $\mathfrak{k}$ of Hermitian matrices.  Every Hermitian
matrix $A$ has a unique set of real eigenvalues, which we denote by $\epsilon(A)$.  Then the group $K$ of unitary 
matrices acts on $\mathfrak{k}$ by conjugation, and $\epsilon$ is constant on each conjugacy class.
A set of eigenvalues can be identified with a point in the dominant Weyl chamber of the Cartan algebra $\mathfrak{h}_+$ of $\mathfrak{k}$,
and therefore we get a surjective map $\epsilon: \mathfrak{k}/K \rightarrow \mathfrak{h}_+$.  The additive eigenvalue
problem is as follows: for which sets of eigenvalues $\mu_1, \ldots, \mu_n \in \mathfrak{h}_+$
 do there exist Hermitian matrices $A_1, \ldots, A_n$ such that $\epsilon(A_i) = \mu_i$ and 
$A_1 + \cdots + A_n = 0$?

We can generalize this to any Lie type as follows.  Let $B \subseteq G$ be a Borel subgroup
containing a maximal torus $T$, and let $K \subseteq G$ be a maximal compact subgroup such that $i \mathfrak{h}_{\mathbb{R}}$
is the Lie algebra of a maximal torus of $K$, where $\mathfrak{h}_{\mathbb{R}}$ is a real form of 
the Lie algebra $\mathfrak{h}$ of $T$.  Then as for Hermitian matrices one can define an eigenvalue map
$\epsilon: \mathfrak{k}/K \rightarrow \mathfrak{h}_+$, where $K$ acts on its Lie algebra $\mathfrak{k}$ 
by the adjoint action.  The eigenvalue problem is then: for which $\mu_1, \ldots, \mu_n \in \mathfrak{h}_+$
 do there exist $A_1, \ldots, A_n \in \mathfrak{k}$ such that $\epsilon(A_i) = \mu_i$ and $A_1 + \cdots + A_n = 0$?
Fixing $n$, the set of tuples $(\mu_1, \ldots, \mu_n)$ satisfying
this statement forms a full-dimensional convex polyhedral subcone $\econe{n}{G} \subseteq \mathfrak{h}_+^n$ called the
additive eigencone.   

Now recall that for any dominant integral weight $\lambda \in \Lambda_+$ of $G$, there is a unique associated irreducible
finite dimensional representation $V_\lambda$.  Consider the following problem: for a tuple of such
weights $(\lambda_1, \ldots, \lambda_n)$, when does the tensor representation $V_{\lambda_1} \otimes \cdots \otimes V_{\lambda_n}$
have a non-trivial invariant subspace $\mathbb{A}_{\vec{\lambda}} = (V_{\lambda_1} \otimes \cdots \otimes V_{\lambda_n})^G$?  Remarkably,
there is a very close connection between this problem and the additive eigenvalue problem.  For any dominant integral weight
$\lambda$ and integer $N>0$, we can scale $\lambda$ by $N$ to get another dominant integral weight $N\cdot \lambda$.
Then the \emph{saturated tensor problem} is: for which tuple of weights $(\lambda_1, \ldots, \lambda_n)$
does there exist a positive integer $N$ such that $\mathbb{A}_{N\vec{\lambda}} \neq \{0\}$?
This problem is equivalent to the additive eigenvalue problem:
letting $\Gamma_n(G)$ be the semigroup of all tuples of weights satisfying the saturated tensor problem, we have that
$(\lambda_1, \ldots, \lambda_n) \in \Gamma_n(G)$ if and only if $(\kappa(\lambda_1), \ldots, \kappa(\lambda_n)) \in \econe{n}{G}$,
where $\kappa: \mathfrak{h}^* \xrightarrow{\sim} \mathfrak{h}$ is the isomorphism induced by
the Killing form.

\subsection{The facets of the additive eigencone}

Since the additive eigencone is a polyhedral cone, it is defined by a unique set of irredundant linear inequalities.
The inequalities are parametrized by products in the cohomology ring of the flag varieties $G/P$, where $P$ is 
a maximal parabolic.  In type A these are the complex Grassmannians.  Let us begin by reviewing the general type
combinatorics of the cohomology of $G/P$.

For any flag variety $G/P$ there is a canonical cell decomposition into \emph{Schubert cells}.  The Schubert cells
are parametrized by cosets in $W/W_P$, where $W$ is the Weyl group of $G$, and $W_P$ is the Weyl group
of $P$.  These cosets each have a unique (minimal length) representative, the set of which
is denoted $W^P$. We denote by $C_w$ the Schubert cell corresponding to $w \in W^P$, and by $\sigma_w \in \text{H}^*(G/P)$
the Poincar\'{e} dual of the homology class of $C_w$.  It is well known that the cohomology
ring $\text{H}^*(G/P)$ is generated by the Schubert classes $\sigma_w$, and therefore the cohomology ring
is determined by the set of positive numbers $c_{u,v}^w$ such that
\[
\sigma_{u} \cdot \sigma_{v} = \sum_{w \in W^P } c_{u,v}^{w} \cdot \sigma_w^*.
\]

Inequalities that determine the additive eigencone are parametrized by cohomology products equal to a multiple of
a point.  The following theorem was proven in type A by Klyachko, and in general type by Berenstein and 
Sjamaar.
\begin{thm}\label{EigIneqThmMult}\normtext{\cite{KLY98,BS00}}
A tuple of dominant weights $\vec{\lambda}$ lies in the eigencone $\econe{n}{G}$ if and only if
for every non-zero cohomology product $\sigma_{w_1} \cdots \sigma_{w_n} = m \normtext{[pt]}$, the 
following inequality is satisfied:
\[
\sum_{i=1}^n \langle \omega_P, w_i^{-1} \lambda_i \rangle \leq 0.
\]
\end{thm}

While these inequalities indeed determine the multiplicative polytope, they are not irredundant.  
The facets (codimension-one faces) intersecting the interior of the dominant chamber correspond in 
general to a subset of the of the above inequalities.  Kapovich, Leeb, and Millson showed (building
on the work of Klyachko, Berenstein-Sjamaar, and additionally Belkale \cite{BEL01}) that the
list of inequalities can be reduced to those associated to products multiplying to a single point.
However in general this list of inequalities is still not irredundant.
The solution is to restrict to \emph{Levi-movable products} \cite{BK06}.  

A Levi-movable intersection is defined as follows.  Let $\Lambda_w = w^{-1} C_w$.  Then
if $\sigma_{w_1} \cdots \sigma_{w_n} = m \text{[pt]}$, there exists, by Kleiman's transversality theorem, 
generic $p_1, \ldots p_n \in P$ such that $\bigcap_i p_i \Lambda_i$ is transverse at the identity $\overline{e} \in G/P$.
The intersection is Levi-movable if we can find $l_1, \ldots, l_n$ in the Levi subgroup $L$ of $P$ such that
$\bigcap_i l_i \Lambda_i$ is transverse at $\overline{e}$.  We say that a cohomology product
is Levi-movable if the corresponding intersection of Schubert varieties is Levi-movable. 

Belkale and Kumar showed that Levi-movability is an algebraic condition and can be expressed
completely in terms of weights of $G$.  For any $w \in W^P$, define $\chi_w$ as
\[
\chi_w = \sum_{\beta \in (R^+ \setminus R^+_L) \cap w^{-1} R^+} \beta.
\]
Alternatively, $\chi_w$ can be shown (see \cite[1.3.22.3]{KSKAC}) to be equal to $\rho - 2\rho^L + w^{-1} \rho$, where $\rho$ and $\rho^L$ are one-half
the sums of the positive roots of $G$ and $L$, respectively.  Also, let $x_1, \ldots, x_r \in \mathfrak{h}$ be the
dual basis of $\alpha_1, \ldots, \alpha_r$.  Then Belkale and Kumar proved the following theorem.
\begin{thm}\normtext{\cite{BK06}}
Suppose that $\sigma_{w_1} \cdots \sigma_{w_n} = m\normtext{[pt]}$ and is non-zero.  Then the following inequality holds
\[
(\chi_1 -  \sum_{i=1}^n \chi_{w_i})(x_P) \geq 0
\]
and the product is Levi-movable if and only if this inequality is satisfied with equality.
\end{thm}

Belkale and Kumar  proved that the set of inequalities corresponding to Levi-movable products are 
sufficient to determine the multiplicative polytope \cite{BK06}.  Ressayre then proved that these inequalities are irredundant in 
\cite{RSS10}.  These results are summarized in the following theorem.

\begin{thm}\label{EigIneqThm}\normtext{\cite{BK06,RSS10}}
A tuple of dominant weights $\vec{\lambda}$ lies in the eigencone $\econe{n}{G}$ if and only if
for every Levi-movable cohomology product $\sigma_{w_1} \cdots \sigma_{w_n} = \normtext{[pt]}$, the 
following inequality is satisfied:
\[
\sum_{i=1}^n \langle \omega_P, w_i^{-1} \lambda_i \rangle \leq 0.
\]
These inequalities are irredundant.
\end{thm}

As a corollary, these theorems provide a method for proving a subgroup $M \subseteq G$ induces a sub-eigencone.
\begin{cor}\label{SubconeCor}
Suppose that a subgroup $M \subseteq G$ induces an isometric embedding $\econe{n}{M} \rightarrow \econe{n}{G}.$
Then $M$ induces a sub-eigencone if the following condition is satisfied: for every maximal parabolic $P \subseteq G$,
maximal parabolic $Q = M \cap P$, and
every Levi-movable product $\sigma_{w_1}^M \cdots \sigma_{w_n}^M = \normtext{[pt]}$ over $M/Q$, the corresponding product 
$\sigma_{w_1} \cdots \sigma_{w_n}$ over $G/P$ is non-zero and equal to $m\normtext{[pt]}$ for some $m > 0$.
\end{cor}
\begin{proof}
Assume that $\vec{\lambda} = (\lambda_1, \ldots, \lambda_n)$
are weights of $M$, and that $\vec{\lambda} \in \econe{n}{G}$.  By Theorem \ref{EigIneqThm}
we need to show that $\vec{\lambda}$ satisfies inequalities corresponding to Levi-movable cohomology products
$\sigma_{w_1}^M \cdots \sigma_{w_n}^M = \text{[pt]}$.  But by assumption, the product
$\sigma_{w_1} \cdots \sigma_{w_n}$ over $G/P$ is non-zero, and therefore by Theorem \ref{EigIneqThmMult}, $\vec{\lambda}$
must satisfy the corresponding inequality: $\sum_i \langle \omega_P, w_i^{-1} \lambda_i \rangle \leq 0$.  
But since the Killing form and Weyl group action of $M$ is preserved by the inclusion of dominant chambers
$\mathfrak{h}_{+,M} \subseteq \mathfrak{h}_{+,G}$, 
this is the same as the inequality associated to $\sigma_{w_1}^M \cdots \sigma_{w_n}^M$, finishing the
proof.
\end{proof}

\subsection{The eigencone and cohomology in types B and C}

The additive eigencone is determined by the Coxeter system associated to $G$, that is, the Weyl group
$W$ of $G$ together with its action on the weight space $\mathfrak{h}^*$ \cite{KLMJ09}.  This implies that 
the eigencones of $\text{Sp}(2r)$ and $\text{SO}(2r+1)$ are isomorphic since their Coxeter systems 
are identical, even though their root systems are not.

Let $G=\text{Sp}(2r)$ and $H=\text{SO}(2r+1)$.  The isomorphism of the of the Weyl groups of $G$ and $H$
induces a correspondence between Schubert cells of $G/P$ and $H/P^H$, where $P$ is a maximal
parabolic of $G$ and $P^H$ is the corresponding maximal parabolic of $H$.  In fact, the cohomology rings
of $G$ and $H$ are isomorphic via a graded isomorphism (see the appendix of \cite{BKISO}). Then the isomorphism
of the eigencones of $G$ and $H$ implies the following proposition.

\begin{prop}\label{BCProdDual}
For every Levi-movable product $\sigma_{w_1} \cdots \sigma_{w_n} = 1\cdot\normtext{[pt]}$, the corresponding product 
$\sigma_{w_1}^H \cdots \sigma_{w_n}^H$ over $H/P^H$ is equal to $1\cdot\normtext{[pt]}$ and is Levi-movable.
\end{prop}
\begin{proof}
Since $\sigma_{w_1} \cdots \sigma_{w_n} = 1[pt]$ is Levi-movable, it corresponds to
a regular face of the eigencone, or in other words the corresponding inequality is irredundant.  
Since the cohomology rings of $G/P$ and $H/P^H$ are isomorphic, the product $\sigma_{w_1}^H \cdots \sigma_{w_n}^H$
is non-zero, and therefore corresponds to an inequality a point in
the eigencone of $H$ must satisfy.  But this inequality is the same as the inequality corresponding
to $\sigma_{w_1} \cdots \sigma_{w_n}$, and therefore defines a regular face of the eigencone.  Since 
the inequalities coming from products $\sigma_{w_1}^H \cdots \sigma_{w_n}^H = m[pt]$ are pairwise distinct
even up to scalars (see the beginning of section 8 in \cite{BK}), the product in question must be Levi-movable and equal
to a point with multiplicity one.
\end{proof}

\section{Cohomology of isotropic Grassmannians}\label{PrelimSection}

In this section we introduce the subgroup $M \cong \text{Sp}(2s) \times \text{Sp}(2(r-s))$ of $G=\text{Sp}(2r)$, and compare 
the codimensions of Schubert cells in flag varieties associated to $M$ and $G$.  Our main result in this section is Proposition 
\ref{ExpDimProp} which gives the expected dimension of intersections of Schubert varieties in $G/P$ associated
to cohomology products in $M/Q$ that parametrize the regular facets of the eigencone of $M$.

\subsection{Preliminaries on $G$ and $M$}

Let $\alpha_1, \ldots, \alpha_r \in \mathfrak{h}^*$ be the simple roots of $G$, where $\mathfrak{h}$ is a Cartan algebra of
$\mathfrak{g} = Lie(G)$, and where we have chosen a Borel subgroup $B \subseteq G$, fixing the ordering of the roots.
Let $V=\mathbb{C}^{2r}$.  We fix an ordered basis $e_1, e_2, \ldots, e_r, e_r', e_{r-1}', \ldots, e_1'$,
and let $(,)$ be the non-degenerate symplectic form satisfying
\begin{enumerate}
\item $(e_i, e_j') = \delta_{ij}$
\item $(e_i, e_j) = 0$
\item $(e_i', e_j') = 0$.
\end{enumerate}
Then a maximal torus $T$ of $G$ is given by matrices of the form $\text{diag}(a_1, a_2, \ldots, a_r, a_r^{-1}, \ldots, a_1^{-1})$.
Let $\epsilon_1, \ldots, \epsilon_r$ denote the characters of $T$ such that $\epsilon_i(\text{diag}(a_1, \ldots, a_i, \ldots)) = a_i$.  These characters correspond to an orthonormal basis of $\mathfrak{h}^*$ with respect to the Killing form $\langle, \rangle$.
Then following Bourbaki \cite[VI.4.6]{BLIE46},
the roots of $G$ are $\pm \epsilon_i \pm \epsilon_j \neq 0$, and the positive roots are 
$\epsilon_i - \epsilon_j$ for $1 \leq i < j \leq r$ and $\epsilon_i + \epsilon_j$ for $1 \leq i \leq j \leq r$.  
The simple roots of $G$ are then $\epsilon_i - \epsilon_{i+1}$ for $1 \leq i < r$ and $2\epsilon_r$.  
The fundamental weights $\omega_1, \ldots, \omega_r$ of $G$ are given by $\omega_i = \epsilon_1 +\epsilon_2 + \cdots + \epsilon_i$.

Let $\tau$ be the element of $T$ corresponding
to $\text{diag}(-1, \ldots, -1, 1, \ldots, 1, -1, \ldots, -1)$, where there are $2(r-s)$ positive ones.  Let 
$M = C_G(\tau) \cong \text{Sp}(2s) \times \text{Sp}(2(r-s))$, where $C_G(\cdot)$ denotes the centralizer in $G$.  Then 
the roots of $M$ can be identified with the roots $\alpha$ of $G$ such that $\alpha(\tau) = 1$.  The simple roots $\beta_i$ of
$M$ are therefore
\[
\epsilon_1 - \epsilon_2, \ldots, \epsilon_{s - 1} - \epsilon_{s}, 2\epsilon_{s}
\]
for the first factor and
\[
\epsilon_{s+1} - \epsilon_{s+2}, \ldots, \epsilon_{r-1} - \epsilon_{r}, 2\epsilon_r 
\]
for the second factor.  In terms of the simple roots of $G$ these are
\[
\alpha_1, \ldots, \alpha_{s-1}, 2\alpha_{s} + 2\alpha_{s+1} + \cdots + 2\alpha_{r-1} + \alpha_r
\]
and
\[
\alpha_{s+1}, \ldots, \alpha_{r},
\]
respectively.  

It is easy to see that the fundamental weights $\nu_1, \ldots \nu_r$ of $M$  are given by 
$\nu_i = \omega_i$ for $1 \leq i \leq s$ and $\nu_i = \omega_i - \omega_s$ for $s+1 \leq i \leq r$.  
So clearly if $\lambda = \sum_i a_i \omega_i$ then
\[
\lambda = \sum_{i=1}^{s-1} a_i \nu_i + (a_s + \cdots + a_r)\nu_s + \sum_{i=s+1}^r a_i \nu_i.
\]
Therefore the dominant chamber of $G$ is contained in the dominant chamber of $M$, but they are not identical.
However, the first factor $M_1 = \text{Sp}(2s)$ of $M$ has a dominant chamber contained in the dominant chamber of $G$.

\subsection{Weyl groups of $G$ and $M$ and the codimension of Schubert varieties}

We need to understand the relationship between the Weyl groups of $G$ and $M$ in order to understand the relationship
between certain cohomology products which parametrize the faces of the eigencones.  To begin with, note that the Cartan
algebras of $G$ and $M$ can be identified, and furthermore a common (normalized) Killing form can be chosen.  Therefore
the root system of $M$ can be identified as above as a sub-root system of the root system of $G$, and therefore the Weyl
group $W_M$ of $M$ is contained in the Weyl group $W$ of $G$.  Now let $s_1, \ldots, s_r$ be the simple reflections generating
$W$, and $t_1, \ldots, t_r$ the simple reflections of $W_M$.  Then it is easy to see that $t_i = s_i$ in $W$ for $i \neq s$, and
$t_s = s_s s_{s+1} \cdots s_{r-1} s_r s_{r-1} \cdots s_s$.  Therefore:
\begin{lemma}
The Weyl group $W_M$ is identified with the subgroup of $W$ generated by the simple reflections \\
$s_1, \ldots, s_{s-1}, s_{s+1}, \ldots, s_r$ and the reflection $s_s \cdots s_{r-1} s_r s_{r-1} \cdots s_s$,
which corresponds to the orthogonal reflection with respect to $2\alpha_{s} + 2\alpha_{s+1} + \cdots + 2\alpha_{r-1} + \alpha_r$.
\end{lemma}

We are interested in the relationship between the cohomology of $G/P \cong \text{IG}(k, r)$ and $M/Q$, where $P$ is a maximal parabolic and
$Q = M \cap P$.  In particular we are interested in the cases where $k \leq s$, so that
we have  $M/Q \cong \text{IG}(k, s)$, and $M/Q$ is identified with the $k$-dimensional isotropic
subspaces in $\text{IG}(k, r)$ contained in a certain $s$-dimensional subspace of $\mathbb{C}^{2n}$ (see section \ref{ProperSec}).
The Schubert cells of these spaces are parametrized by cosets in $W/W_P$ and $W_M/W_Q$, which have minimal length
representatives; the sets of minimal length representatives are denoted $W^P$ and $W_M^Q$, respectively.  
Then we have the following proposition.

\begin{prop}
The inclusion $W_M \subseteq W$ induces an inclusion $W_M^Q \subseteq W^P$ where $P$ is chosen as above.
\end{prop}
\begin{proof}
Let $I$ be the set of simple reflections generating $W_P$, and $I_M$ be the set of simple reflections (simple in $W_M$)
generating $W_Q$.  In Chapter 1 of \cite{HUM}, Humphreys shows that $W^P$ is equal to the set
\[
W^I = \{ w \in W \mid \ell(ws_\alpha) > \ell(w) \text{ for all } s_\alpha \in I \}.
\]
Also he shows that $\ell(ws_\alpha) > \ell(w)$ exactly when $w \alpha > 0$.  Suppose $w \in W_M^Q$.  Now $W_M$ is actually a product of the
Weyl groups of the two factors of $M = M_1 \times M_2$, and by the choice of $P$, any minimal representative in $W_M^Q$
has a trivial $M_2$ part.  Therefore $w$ acts trivially on any $\alpha_i$ for $i>s$, and therefore $w \alpha_i >0$
in this case.  For a simple reflection $s_{\alpha_i} \in I$ where $i < s$, we have $w \alpha_i >0$ by the assumption that $w \in W_M^Q$,
since $s_{\alpha_i} \in I_Q$.  Now consider $s_{\alpha_s} \in I$.  Then letting 
$\beta_s = 2\alpha_{s} + 2\alpha_{s+1} + \cdots + 2\alpha_{r-1} + \alpha_r$, we know that $s_{\beta_s} \in I_Q$ and
$w\beta_s > 0$ with respect to the ordering of the first factor.  But 
$w(2\alpha_{s+1} + \cdots + 2\alpha_{r-1} + \alpha_r) = 2\alpha_{s+1} + \cdots + 2\alpha_{r-1} + \alpha_r$, and therefore if
$w \alpha_s$ is negative, it must only contain roots $\alpha_i$ with $i > s$.  This is impossible.  Therefore $w \alpha_s > 0$
and so $w \in W^P$.
\end{proof}

Now denote the Schubert cell in $G/P$ associated to $w \in W^P$ as $C_w$, and the Schubert cell in $M/Q$ 
associated to $w \in W_M^Q$ as $C_w^M$.  One important property of the relationship between the cell decompositions
of $G/P$ and $M/Q$ is that for $w \in W_M^Q$, the dimension of $C_w$ may be higher than the dimension of $C_w^M$.
For example, in the case of the big cell in $M/Q$, the jump is the difference in dimension between $G/P$ and $M/Q$.

\begin{comment}
\begin{prop}
Let $w_0 \in W_M^Q$ correspond to the big cell $C_{w_0}^M$ in $M/Q$.  Then $C_{w_0}$ is the big cell
in $G/P$.
\end{prop}
\begin{proof}
We know that $w_0^{-1} R^+_M \cap R^+_M \setminus R^+_L = \varnothing$, and we need to show the same for
$w_0^{-1} R^+ \cap R^+\setminus R^+_L$.  Assume  that  $\alpha \in R^+\setminus R^+_L$ and  $w_0\alpha \in R^+$.  We can
furthermore assume that $\alpha \notin R^+_M$. Then as above, $w_0$ 
will act trivially on the roots of the second factor of $M$, so by adding roots from this factor we can assume that
$\alpha = \alpha_m + \alpha_{m+1} + \cdots + \alpha_s + 2\alpha_{s+1} + \cdots + 2\alpha_{r-1} + \alpha_r$ for some $m$.
For the same reason, the root 
$\beta = \alpha_m + \alpha_{m+1} + \cdots + \alpha_s$ satisfies $w_0 \beta \in R^+$.  But then $\alpha + \beta \in R^+_M$
and $w_0(\alpha + \beta) \in R^+$, which contradicts the fact that $w_0^{-1} R^+_M \cap R^+_M \setminus R^+_L = \varnothing$.
\end{proof}
\end{comment}

In general we can calculate the difference in dimension in terms of certain characters associated to $w$.  
As in section \ref{EigenconeSec} let
\[
\chi_w = \sum_{\alpha \in w^{-1}R^+ \cap R^+ \setminus R^+_L} \alpha
\]
where $R^+_L$ is the set of positive roots of the Levi $L$ of $P$. Let $\chi_w^M$ be defined in the same way, 
but only including roots of $M$.  Also, let $x_1, \ldots, x_r \in \mathfrak{h}$ be the dual
basis of $\alpha_1, \ldots, \alpha_r$.  Then we have the following.

\begin{prop}\label{ChiProp}
For $w \in W_M^Q$ we have
\[
\normtext{codim}(C_w) - \normtext{codim}(C_w^M) = (\chi_w - \chi_w^M)(x_P)
\]
where $\alpha_p$ is the root corresponding to $P$.
\end{prop}
\begin{proof}
It is a fact (see \cite[Lemma 16]{BK06}) that $\normtext{codim}(C_w) = | w^{-1} R^+ \cap R^+ \setminus R^+_L|$ and similarly 
$\normtext{codim}(C_w^M) = | w^{-1} R^+_M \cap R^+_M \setminus R^+_L| = | w^{-1} R^+ \cap R^+_M \setminus R^+_L|$.
It is easy to see that the coefficient of $\alpha_P$ for the roots in $R^+ \setminus R^+_M$ is always one, and
since $(\chi_w - \chi_w^M)(x_P)$ is equal to the coefficient of $\alpha_P$ in $\chi_w - \chi_w^M$, the result follows.
\end{proof}

Each Schubert cell in $M/Q$ and $G/P$ corresponds to a subset of $\{1, \ldots, 2s\}$ or $\{1, \ldots, 2r\}$,
which indicates how a subspace in the cell intersects a chosen flag.  For example, the class of a point in 
$G/P \cong \text{IG}(k,2r)$ corresponds to $\{1, 2, \ldots, k\}$, and the big cell corresponds to $\{2r-k+1, \ldots, 2r \}$.
For more details see section 4 in \cite{BKT09}.
For a Weyl group element $w \in W_M^Q$,
let $I_w^M$ and $I_w$ denote the subsets corresponding to $C_w^M$ and $C_w$, respectively.  Note
that $I_w$ is constructed from $I_w^M$ by adding $2(r-s)$ to each element $i \in I_w^M$ such that $i>s$.  We make
the following definition.

\begin{defin}
For a set of integers $I$ and an integer $m$, let $|I \leq m|$, $|I \geq m|$, $|I < m|$, and $|I > m|$ denote the number of elements $i \in I$
such that $i \leq m$, $i \geq m$, $i < m$, or $i > m$, respectively.  
If $J$ is another set of integers, let $|I \leq J|$, $|I \geq J|$, $|I < J|$, and
$|I > J|$ be the number of pairs $(i,j)$ such the $i \leq j$, $i \geq j$, $i < j$, or
$i > j$, respectively.
\end{defin}

Then we can also calculate the difference in codimension between $C_w^M$ and $C_w$ in terms
of the associated subsets $I_w^M$ and $I_w$.

\begin{prop}\label{CodimProp2}
For $w \in W_M^Q$ we have
\[
\normtext{codim}(C_w) - \normtext{codim}(C_w^M) = 2(r-s)|I^M_w \leq s|.
\]
\end{prop}
\begin{proof}
We prove the equivalent statement that
\[
\normtext{dim}(C_w) - \normtext{dim}(C_w^M) = 2(r-s)|I^M_w > s|.
\]
(Recall that $\normtext{dim}(\normtext{IG}(k,2r)) = \frac{k}{2}(4r-3k+1)$, so that 
$\normtext{dim}(\normtext{IG}(k,2r)) - \normtext{dim}(\normtext{IG}(k,2s)) = 2k(r-s)$.)
Now Proposition 32 in \cite{BKISO} states that
\[
\text{dim}(C_w) = |I_w > \widetilde{I}_w| + \frac{1}{2}(|I_w > \bar{I}_w| + |I_w > r|),
\]
where if $I_w = \{i_1 < i_2 < \cdots < i_k\}$,
\begin{align*}
\bar{I}_w &= \{2r+1-i_1, 2r+1-i_2, \ldots, 2r+1-i_k\} \\
\widetilde{I}_w &= \{1,2,\ldots, 2r\} \setminus (I_w \sqcup \bar{I}_w).
\end{align*}
The formula for $\text{dim}(C_w^M)$ is similar.
It is immediate that $|I_w > r|= |I_w^M > s|$ and $|I_w > \bar{I}_w| = |I_w^M > \bar{I}_w^M|$.
Then since $I_w$ is constructed from $I_w^M$ by adding $2(r-s)$ to each element $i \in I_w^M$ such that $i>s$,
we see that
\begin{align*}
\normtext{dim}(C_w) - \normtext{dim}(C_w^M) &= |I_w > \widetilde{I}_w| - |I_w^M > \widetilde{I}_w^M| \\
                                                        &= 2(r-s)|I^M_w > s|.
\end{align*}
\end{proof}

The odd orthogonal groups $\text{SO}(2r+1)$ are closely related to the symplectic groups (see section \ref{EigenconeSec}).
In the next section we will use the relationship between the cohomology
of $M$ and $\oddO = \text{SO}(2s+1) \times \text{Sp}(2(r-s))$ to compute the expected dimension of intersections
of Schubert varieties in $G/P$.  We begin by describing the roots of $\oddO$.  

The weight space of $\oddO$ can
be identified with the weight space of $M$, and with respect to this identification the simple roots of $\oddO$
are the same as $M$, except that $\alpha_s$ is replaced with $\frac{1}{2}\alpha_s$.  More precisely, the
roots of the first factor of $\oddO$ are $\pm \epsilon_i \pm \epsilon_j$
for $1 \leq i \neq j \leq s$, and additionally $\pm \epsilon_i$ for $1\leq i \leq s$.
The positive roots are $\epsilon_i \pm \epsilon_j$
for $1 \leq i < j \leq s$, together with $\epsilon_i$ for $1\leq i \leq s$, with the simple roots 
being $\beta_1 = \epsilon_1 - \epsilon_2, \beta_2 = \epsilon_2 - \epsilon_3, \ldots, \beta_{s-1} = \epsilon_{s-1} - \epsilon_s, \beta_s=\epsilon_s$.
The simple roots of the second factor are the same as for $M$, but we will denote them using $\beta$
when referring to $\oddO$.

The two groups $M$ and $H$ clearly then have the same Weyl groups, which we denote $W_M$.  Furthermore, choosing a parabolic $Q$
of $M$, the set $W_M^Q$ is the same as the set of minimal length representatives for the corresponding parabolic $Q^\oddO$ of $\oddO$.
Let 
\[
\chi_w^\oddO = \sum_{\beta \in w^{-1}R^+_\oddO \cap R^+_\oddO \setminus R^+_L} \beta.
\]
Then we have the following proposition.

\begin{prop}\label{ChiHProp}
For any $w \in W_M^Q$ we have
\[
(\chi^M_w - \chi_w^\oddO)(x_Q) = |I_w^M \leq s|.
\]
\end{prop}
\begin{proof}
Let $\rho^M$ be one-half the sum of the positive roots of $M$, and $\rho^M_L$ be one-half the sum of the positive roots
of $Q$.  Then it is a fact (see section 3 of \cite{BK06}) that 
\[
\chi_w^M = \rho^M - 2\rho^M_L + w^{-1}\rho^M.
\]
Let $\rho^\oddO$ and $\rho^\oddO_L$ be one-half the sum of the positive roots of $\oddO$ and $Q^\oddO$, respectively.
Then clearly $\rho^M - \rho^\oddO = \frac{1}{2}\sum_{i=1}^s \epsilon_i$ and 
$\rho^M_L - \rho^\oddO_L = \frac{1}{2}\sum_{i=k+1}^s \epsilon_i$.  Then writing $\epsilon = \sum_{i=1}^s \epsilon_i$
and $\epsilon^L = \sum_{i=1}^k \epsilon_i - \sum_{i=k+1}^s \epsilon_i$, we see that
\[
\chi_w^M - \chi_w^\oddO = \frac{1}{2}(\epsilon^L + w^{-1}\epsilon).
\]
It is easy to see that $x_Q = \sum_{i=1}^k \epsilon^*_i$, where $\epsilon_1^*, \ldots, \epsilon_r^*$
is the dual basis, so that $\epsilon^L(x_Q) = k$.  Therefore it is
sufficient to show that $w^{-1}\epsilon(x_Q) = k-2|I^M_w > s|$.  Now Lemma 39 of \cite{BKISO} implies that 
$\langle w\epsilon, \epsilon \rangle = s-2|I^M_w > s|$, so then
\begin{align*}
w^{-1}\epsilon(x_Q) &= \langle w^{-1} \epsilon, \epsilon \rangle - \langle w^{-1} \epsilon, \sum_{i=k+1}^s \epsilon_i \rangle \\
            &= s-2|I^M_w > s| - \langle \epsilon, w \sum_{i=k+1}^s \epsilon_i \rangle.
\end{align*}
Now it is easy to see that $w \sum_{i=k+1}^s \epsilon_i = \sum_{i=1}^{s-k} \epsilon_{j_i}$ for some
$1\leq j_1 < j_2 < \cdots < j_{s-k} \leq s$ (see Lemma 19 of \cite{BKISOAXV}), and
therefore that $\langle \epsilon, w \sum_{i=k+1}^s \epsilon_i \rangle = s-k$, finishing the proof.
\end{proof}

\subsection{Expected dimension of intersections of Schubert varieties in $G/P$}

Our goal is to understand the relationship between Levi-movable products $\sigma_{w_1}^M \cdots \sigma_{w_n}^M = 1\cdot\text{[pt]}$ 
in $\text{H}^*(M/Q)$ and the associated product $\sigma_{w_1} \cdots \sigma_{w_n}$ in $\text{H}^*(G/P)$.  In this section we 
prove that the product over $G/P$ has degree zero, or equivalently that the expected dimension of the corresponding
intersection of Schubert varieties is zero.
We need to define some notation.
\begin{defin}
For any tuple of minimum length representatives $\vec{w} = (w_1, \ldots, w_n)$ we write $\theta(\vec{w}) = (\chi_1 - \sum_{i=1}^n \chi_{w_i})(x_P)$, 
$\theta^M(\vec{w}) = (\chi^M_1 - \sum_{i=1}^n \chi^M_{w_i})(x_Q)$,
and $\theta^H(\vec{w}) = (\chi^H_1 - \sum_{i=1}^n \chi^H_{w_i})(x_Q)$.  A (non-zero) cohomology product corresponding to
$\vec{w}$ is Levi-movable if the appropriate number $\theta(\vec{w})$, $\theta^M(\vec{w})$, or $\theta^H(\vec{w})$
is zero.
For the \emph{expected dimension} of $\bigcap_i C_{w_i}$ in $G/P$ we write
\[
\expDim{\vec{w}}{G} = \text{dim}(G/P) - \sum_{i=1}^n \text{codim}(C_{w_i}).
\]
The expected dimension of $\bigcap_i C_{w_i}^M$ is written $\expDim{\vec{w}}{M}$.
\end{defin}

The following lemmas relate $\theta(\vec{w})$, $\theta^M(\vec{w})$, and $\theta^H(\vec{w})$ to the
difference in expected dimension of intersections in $M/Q$ and $G/P$.

\begin{lemma}
For any $\vec{w} = (w_1, \ldots, w_n)$ where $w_i \in W^Q_M$ for each $i$, we have
\[
\theta(\vec{w}) - \theta^M(\vec{w}) = \expDim{\vec{w}}{G} -\expDim{\vec{w}}{M}.
\]
\end{lemma}
\begin{proof}
By Proposition \ref{ChiProp}, and we have
\begin{align*}
\theta(\vec{w}) - \theta^M(\vec{w}) & = (\chi_1 - \chi_1^M)(x_P) - \sum_{i=1}^n (\chi_{w_i} - \chi_{w_i}^M)(x_P) \\
                                    & = \text{dim}(G/P) - \text{dim}(M/Q) - \sum_{i=1}^n \text{codim}(C_{w_i}) - \text{codim}(C_{w_i}^M) \\
                                    & = \expDim{\vec{w}}{G} -\expDim{\vec{w}}{M}.
\end{align*}
\end{proof}

\begin{lemma}
For any $\vec{w} = (w_1, \ldots, w_n)$ where $w_i \in W^Q_M$ for each $i$, we have
\[
\theta^M(\vec{w}) - \theta^H(\vec{w}) = \frac{1}{2(r-s)}(\expDim{\vec{w}}{G} -\expDim{\vec{w}}{M}).
\]
\end{lemma}
\begin{proof}
By Propositions \ref{ChiHProp} and \ref{CodimProp2}, we have
\begin{align*}
\theta^M(\vec{w}) - \theta^H(\vec{w}) & = (\chi^M_1 - \chi_1^H)(x_Q) - \sum_{i=1}^n (\chi^M_{w_i} - \chi_{w_i}^H)(x_Q) \\
                                    & = |I_1^M \leq s| - \sum_{i=1}^n |I_{w_i}^M \leq s| \\
                                    & = \frac{1}{2(r-s)}\left(\text{dim}(G/P) - \text{dim}(M/Q) - \sum_{i=1}^n \text{codim}(C_{w_i}) - \text{codim}(C_{w_i}^M)\right) \\
                                    & = \frac{1}{2(r-s)}(\expDim{\vec{w}}{G} -\expDim{\vec{w}}{M}).
\end{align*}
\end{proof}

The following is the main proposition of the section.

\begin{prop}\label{ExpDimProp}
If $\sigma_{w_1}^M \cdots \sigma_{w_n}^M = 1[pt]$ is Levi-movable, then $\expDim{\vec{w}}{G} = \theta(\vec{w}) = 0$.
\end{prop}
\begin{proof}
By hypothesis, $\theta^M(\vec{w}) = \expDim{\vec{w}}{M} = 0$.  Therefore it suffices by the above lemmas
to show that $\theta^H(\vec{w}) = 0$.  By Proposition \ref{BCProdDual} the product $\sigma_{w_1}^H \cdots \sigma_{w_n}^H$
is non-zero and Levi-movable, so the proposition follows.
\end{proof}

\section{Properness of cohomology products in $G/P$ and proof of the eigencone theorems}\label{ProperSec}

In this section we assume that for Weyl group elements $w_i \in W^Q_M$, we have 
$\sigma_{w_1}^M \cdots \sigma_{w_n}^M = 1[pt]$, and furthermore this product is Levi-movable.  We furthermore
assume that $s = r-1$.  This case is sufficient to prove the eigencone result.  The main
theorem in this section is that the corresponding cohomology product in $\text{H}^*(G/P)$ is equal
to a positive multiple of the class of a point (Theorem \ref{CohomThm}).

The idea behind the proof is as follows.  We know from Proposition \ref{ExpDimProp} that
if $\sigma_{w_1} \cdots \sigma_{w_n} \neq 0$ then the theorem follows.  By \cite[Prop. 7.1 and Sec. 12.2]{FULINT}
and Kleiman transversality (see \cite[Prop. 1.1]{B06}) it follows that
if an intersection of Schubert varieties is non-empty and proper for some choice of flags -- that is, the
intersection is the expected dimension -- then the corresponding cohomology product is non-zero.  Our strategy is 
to move the Schubert varieties in $G/P$ using $M$ in order to make the intersection proper.  

If $M$ acted transitively on $G/P$ this would be trivial, and the intersection would be non-zero
because the corresponding intersection in $M/Q$ is non-zero.  While $M$ does not act transitively
on $G/P$, it does act with a finite number of orbits. This allows us to check the properness
of the intersection orbit-by-orbit and conclude that if the intersection is the expected dimension
in each orbit, then the intersection is proper.  Since $M$ acts transitively on each orbit,
checking the dimension of the intersection within each orbit amounts to computing the dimension of 
each Schubert variety within each orbit.  This is the strategy which was used successfully in 
\cite{BKISOAXV} to prove similar cohomology results for $\text{SO}(2r+1)$ and $\text{Sp}(2r)$ in
$\text{SL}(2r+1)$ and $\text{SL}(2r)$, respectively.  Similar methods were also used in \cite{BKT09}
to prove quantum Pieri rules for isotropic Grassmannians.

Following this strategy leads to the conclusion that some intersections will always be above the
expected dimension if we only move the Schubert varieties using $M$.  However, by first
shifting one of the varieties by a particular element of $G$, and then 
allowing shifting by general elements of $M$, these intersections can be made proper and non-empty.

\subsection{Action of $M$ on $G/P$}

First we give a description of the action of $M$ on $G/P$, including the orbits.  We will
see that this action has four orbits, which we denote  $\mathcal{O}_1$, $\mathcal{O}_2$, $\mathcal{O}_2'$, 
and $\mathcal{O}_3$.

Let $V=\mathbb{C}^{2r}$.  As in section \ref{PrelimSection}, we fix an ordered basis $e_1, e_2, \ldots, e_r, e_r', e_{r-1}', \ldots, e_1'$,
and let $(,)$ be the non-degenerate symplectic form satisfying
\begin{enumerate}
\item $(e_i, e_j') = \delta_{ij}$
\item $(e_i, e_j) = 0$
\item $(e_i', e_j') = 0$.
\end{enumerate}
Let $E_\bullet$ be the full flag of $V$ corresponding to this ordered basis.  We will identify $\overline{e} \in G/P$
with the isotropic subspace $W_0 = \text{span}(e_1, e_2, \ldots, e_{k})$.

Let $V_1 = \text{span}(e_1, e_2, \ldots, e_{r-1}, e_{r-1}', \ldots, e_1')$, and $V_2 = \text{span}(e_r, e_r')$.
Then $M$ is the subgroup of $G$ fixing $V_1$ and $V_2$.  Clearly then $M/Q \cong \text{IG}(k,V_1)$, where
$V_1$ is equipped with the form $(,)$ restricted from $V$.  The inclusion $M/Q \subseteq G/P$ is
simply given by the inclusion $V_1 \subseteq V$.  We can describe the orbits of $M$ in terms of the dimensions
of these spaces: 

\begin{defin}
For any isotropic subspace $W\subseteq V$ of dimension $k$, let $W_1$ be the orthogonal projection of $W$
to $V_1$, and $W_2$ the projection to $V_2$.  Furthermore:
\begin{enumerate}
\item Let $\mathcal{O}_1$ be the subset of $G/P$ such that $\text{dim}(W_2) = 0$;
\item Let $\mathcal{O}_2$ be the subset of $G/P$ such that $\text{dim}(W_2) = 1$ and $\text{dim}(W_1) = k$; 
\item Let $\mathcal{O}_2'$ be the subset of $G/P$ such that $\text{dim}(W_2) = 1$ and $\text{dim}(W_1) = k-1$; 
\item Let $\mathcal{O}_3$ be the subset of $G/P$ such that $\text{dim}(W_2) = 2$.
\end{enumerate}
\end{defin}
It is immediate that these sets are disjoint and cover all of $G/P$.  Furthermore, since the action of $M$
commutes with the projections $\text{pr}_1: V \rightarrow V_1$ and $\text{pr}_2: V \rightarrow V_2$,
$M$ preserves these subsets.  It remains to show that $M$ acts transitively on these subsets.

\begin{prop}
The sets $\mathcal{O}_1$, $\mathcal{O}_2$, $\mathcal{O}_2'$, and $\mathcal{O}_3$ correspond
to the closed points of the orbits of the action of $M$ on $G/P$.  Furthermore, we have:
\begin{itemize}
\item $\normtext{dim}(\mathcal{O}_1) = \normtext{dim}(\normtext{IG}(k,2(r-1))) = \frac{k}{2}(4(r-1)-3k+1)$
\item $\normtext{dim}(\mathcal{O}_2) = \normtext{dim}(\normtext{IG}(k-1,2(r-1))) +  \normtext{dim}(\normtext{Gr}(1,2)) + \normtext{dim}(\normtext{Gr}(k,k+1))$
\item $\normtext{dim}(\mathcal{O}_2') = \normtext{dim}(\normtext{IG}(k-1,2(r-1))) +  \normtext{dim}(\normtext{Gr}(1,2))$
\item $\normtext{dim}(\mathcal{O}_3) = \normtext{dim}(\normtext{IG}(k,2r)) = \frac{k}{2}(4r-3k+1)$
\end{itemize}
\end{prop}
\begin{proof}
Let $W$ be a $k$-dimensional isotropic subspace of $V$, and let $W_i^{int} = W \cap V_i$.  
Note that $W_1^{int}$ is exactly the kernel of the projection $W \rightarrow W_2$.
Clearly the action of $M$ is transitive on $\mathcal{O}_1$, since it is simply the image of
$M/Q$ in $G/P$.  If we suppose that $\text{dim}(W_2) = 1$ and $\text{dim}(W_1) = k-1$  then $W_i^{int} = W_i$,
so this orbit is clearly isomorphic to $\text{IG}(k-1,V_1) \times \text{Gr}(1,2)$, and 
$M$ acts on each factor in the obvious way; in particular this action is transitive.

Suppose now that $\text{dim}(W_2) = 1$ and $\text{dim}(W_1) = k$.  Then $\text{dim}(W_1^{int}) = k-1$,
so let $w_1, \ldots, w_{k-1}$ be a basis for $W_1^{int}$
and let $w_k \in W$ be a vector that completes this basis to a basis of $W$.  Then we can
write $w_k = w_k^1 + w_k^2$, where $w_k^1 \in W_1$ and $w_k^2 \in W_2$, and neither of these
vectors are zero.  First we claim that $W_1$ is an isotropic subspace of $V_1$.  It is clearly
sufficient to show that $(w_k^1, w_i) = 0$ for all $i < k$.  Since $W$ is isotropic, $(w_k, w_i) = 0$,
and since $W_1^{int} \subseteq V_1$, clearly $(w_k^1, w_i) = 0$ for all $i < k$, so
$W_1$ is isotropic.  Now suppose we have another isotropic subspace $W'$ in this orbit, with
a basis $w_1', \ldots, w_k'$ constructed as above.  Then there is an element $m_1$ of $\text{Sp}(V_1)$
that maps the basis $w_1, \ldots, w_{k-1}, w_{k}^1$ of $W_1$ to the basis $w_1', \ldots, w_{k-1}', (w_{k}^1)'$
of $W_1'$, and an element $m_2$ of $\text{SL}(2)$ mapping $w_k^2$ to $(w_k^2)'$.
(One can extend these bases to symplectic bases of $V_1$ and $V_2$; the change of basis maps
are then symplectic, giving $m_1$ and $m_2$.)  Together
these give an element $(m_1,m_2)$ of $M$ mapping $W$ to $W'$.  Therefore $M$ acts transitively
on $\mathcal{O}_2$.  Now there is a natural surjective morphism 
$\mathcal{O}_2 \rightarrow \text{IG}(k,V_1) \times \text{Gr}(1,2)$ induced by the projection
maps to $V_1$ and $V_2$.  Fixing the spaces $W_1$ and $W_2$, we can construct a subspace in the
fiber of this morphism by choosing an appropriate isotropic $k$-dimensional subspace of $W_1 + W_2$.  Any such subspace is 
automatically isotropic since $W_1 + W_2$ is isotropic, and so we simply need to choose a
subspace of $W_1 + W_2$ whose projection to $V_1$ is $W_1$, and whose projection to $V_2$ is $W_2$.
This is an open condition in $\text{Gr}(k, W_1 + W_2)$, and so the fibers of the morphism
$\mathcal{O}_2 \rightarrow \text{IG}(k,V_1) \times \text{Gr}(1,2)$ can all be identified with
open subsets of $\text{Gr}(k,k+1)$.

Finally, suppose that $\text{dim}(W_2) = 2$.  This is clearly an open condition, so that this
orbit will be open in $G/P$.  It remains to show that $M$ acts transitively.  Now $\text{dim}(W_1^{int}) = k-2$,
so let $w_1, \ldots, w_{k-2}$ be a basis of $W_1^{int}$, and let $w_{k-1}$ and $w_k$ complete this basis to
a basis of $W$.  First we claim that $\text{dim}(W_1) = k$.  Since $W$ is isotropic, 
\begin{align*}
0 &= (w_{k-1},w_k) \\
  &= (w_{k-1}^1, w_k^1) + (w_{k-1}^2, w_k^2)
\end{align*}
and since $(w_{k-1}^2, w_k^2) \neq 0$, we see that $w_{k-1}^1$ and  $w_k^1$ are non-zero and linearly
independent.  Suppose that $a_1, \ldots, a_k \in \mathbb{C}$ are numbers such that 
$w = a_1 w_1 + \cdots + a_{k-2} w_{k-2} + a_{k-1} w_{k-1}^1 + a_k w_k^1 = 0$.  For the sake of contradiction we 
can assume that one of $a_{k-1}$ and $a_k$ is non-zero; without loss of generality assume $a_{k-1} \neq 0$.  But then 
\begin{align*}
0 = (w,w_k) &= (a_{k-1} w_{k-1}^1 + a_k w_k^1, w_k) \\
            &= (a_{k-1} w_{k-1}^1, w_k^1) \neq 0
\end{align*}
which is a contradiction.  Therefore $a_i=0$ for all $i$ (since $w_1, \ldots, w_{k-2}$ are independent), and so 
the vectors $w_1, \ldots, w_{k-2}, w_{k-1}^1, w_k^1$ are linearly independent.  Now if we have another isotropic
subspace $W'$ in $\mathcal{O}_3$, as above by extending to symplectic bases (after scaling $w_k$ so that
$(w_{k-1}^1, w_k^1) = 1$) we can construct an element $m$ of $M$ 
identifying the chosen bases of $W_1$ and $W_2$ with the bases of $W_1'$ and $W_2'$.  This element will then map $W$ to $W'$, completing
the proof.
\end{proof}

\subsection{Proof of Theorem \ref{CohomThm}}

Let $I^M = \{i_1 < \cdots < i_k\}$ be a $k$-sized subset of $\{1, \ldots, 2(r-1)\}$ corresponding to a Schubert
variety in $M/Q$, and $I$ the corresponding subset of $\{1, \ldots, 2r\}$.  Assume that there is an $i \in I$ such that $i < r$.  Then 
we define a full flag $E_\bullet(I)$ of $V$ as follows.
\begin{itemize}
\item $E_0(I) = \{0\}$
\item $E_{i}(I) = E_{i-1}(I) + \text{span}(e_i)$ if $i \leq r$ and $i \neq i_1$
\item $E_{i}(I) = E_{i-1}(I) + \text{span}(e_i + e_r)$ if $i = i_1$ 
\item $E_i(I) = E_{2r-i}(I)^\perp$ if $i > r$
\end{itemize}
To finish the proof of Theorem \ref{CohomThm}, we need to calculate the dimension of the
intersections of $C_I(E_\bullet)$ and $C_I(E_\bullet(I))$ with each
orbit of $M$.

\begin{prop}\label{SchubOrbDim}
If there is an $i\in I$ such that $i > r+1$ then the dimensions of the intersections of $C_I(E_\bullet)$
with each orbit $\mathcal{O}_1$, $\mathcal{O}_2$, $\mathcal{O}_2'$,  and $\mathcal{O}_3$ are: 
\begin{enumerate}
\item $\normtext{dim}(C_I(E_\bullet) \cap \mathcal{O}_1) =  \normtext{dim}(C_I^M)$
\item $\normtext{dim}(C_I(E_\bullet) \cap \mathcal{O}_2) = \normtext{dim}(C_I^M) + |I^M > r-1| + 1$
\item $C_I(E_\bullet) \cap \mathcal{O}_2' = \emptyset$
\item If non-empty, $\normtext{dim}(C_I(E_\bullet) \cap \mathcal{O}_3)=  \normtext{dim}(C_I)$
\end{enumerate}
If there is furthermore an $i\in I$ such that $i<r$, then letting
$E_\bullet(I)$ be the flag constructed above, we have
\begin{enumerate}
\item $C_I(E_\bullet(I)) \cap \mathcal{O}_1 = \emptyset$
\item $\normtext{dim}(C_I(E_\bullet(I)) \cap \mathcal{O}_2) = \normtext{dim}(C_I^M) + |I^M > r-1|$
\item $C_I(E_\bullet(I)) \cap \mathcal{O}_2' = \emptyset$
\item If non-empty, $\normtext{dim}(C_I(E_\bullet(I)) \cap \mathcal{O}_3) =  \normtext{dim}(C_I)$
\end{enumerate}
\end{prop}
\begin{proof}
The first equality follows from the fact that $C_I(E_\bullet) \cap \mathcal{O}_1 = C_I^M$.
The fact that $C_I(E_\bullet(I)) \cap \mathcal{O}_1 = \emptyset$ follows from the fact that
for any $W \in C_I(E_\bullet(I))$, there is a vector $w \in W$ such that 
$w = a_1 e_1 + \cdots + a_{i_1}( e_{i_1} + e_r)$ with $a_{i_1} \neq 0$, so that $W$ is not a subspace of $V_1$.
The empty intersections with $\mathcal{O}_2'$ follow from the fact that $W \in C_I(E_\bullet)$ or $W \in C_I(E_\bullet(I))$
cannot have a jump in dimension at position $r$ or $r+1$ of the corresponding flags, and any
$W \in \mathcal{O}_2'$ must have a jump in dimension at these positions.  The dimensions
of the intersections with $\mathcal{O}_3$ follow from the fact that $\mathcal{O}_3$ is an
open orbit.

It remains to calculate the dimensions in $\mathcal{O}_2$.  Since there is an $i\in I$ 
such that $i > r+1$ the intersection $C_I(E_\bullet) \cap \mathcal{O}_2$ is nonempty.  So let
$W \in C_I(E_\bullet) \cap \mathcal{O}_2$.  Clearly then $W_1$ intersects the 
standard full flag of $V_1$ according to $I^M$.  Therefore the image of 
$C_I(E_\bullet) \cap \mathcal{O}_2 \rightarrow \text{IG}(k, V_1)$ is contained in
$C_I^M$.  Now fix $W_1 \in C_I^M$ and $W_2 \in \text{Gr}(1,2)$.  We claim
that the fiber in $C_I(E_\bullet) \cap \mathcal{O}_2$ over $(W_1, W_2)$
is a $|I^M > r-1|$-dimensional subspace of $\text{Gr}(k, W_1+W_2)$.  Let $F_\bullet$
be the full flag of $W_1+W_2$ induced by $E_\bullet$.  Note that the subset $I'$
corresponding to the intersection of $W_1+W_2$ with $E_\bullet$ is the set $I^M$ with either $r$
or $r+1$ added.  Therefore, any $k$-dimensional subspace of $W_1+W_2$ that intersects
$F_\bullet$ according to $J = \{1,2,\ldots, |I^M \leq r-1|, |I^M \leq r-1| + 2, \ldots, k+1\}$
is in $C_I(E_\bullet)$; in fact $C_I(E_\bullet) \cap \text{Gr}(k, W_1+W_2) = C_J(F_\bullet)$,
which has dimension $|I^M > r-1|$.  
Then since the fiber in $\mathcal{O}_2$ over $(W_1,W_2)$ is open in $\text{Gr}(k, W_1+W_2)$, the fiber in
$C_I(E_\bullet) \cap \mathcal{O}_2$ over $(W_1, W_2)$ is open in $C_J(F_\bullet)$, and nonempty by
assumption.  Finally, it is easy to see that the fiber in $C_I(E_\bullet) \cap \mathcal{O}_2$ over any pair $(W_1, W_2)$ in 
$C_I^M \times \text{Gr}(1,2)$ is non-empty, and therefore by the above work is $|I^M > r-1|$-dimensional.
Therefore the dimension of $C_I(E_\bullet) \cap \mathcal{O}_2$ is 
$\normtext{dim}(C_I^M) + \text{dim}(\text{Gr}(1,2)) + |I^M > r-1|$.

The calculation of the dimension of $C_I(E_\bullet(I)) \cap \mathcal{O}_2$
is essentially identical, except that for any subspace $W$ in this intersection,
$W_2$ is always the line spanned by $e_r$, thus lowering the dimension by one.
\end{proof}

\begin{proof}[Proof of Theorem \ref{CohomThm}]
We follow the strategy outlined at the beginning of this section.  By Proposition \ref{ExpDimProp}
the expected dimension of the intersection is $0$, and if non-empty the intersection will
be Levi-movable.  Let
$I_1, \ldots, I_n$ be sets corresponding to the Schubert varieties.  There are two
cases.  First assume that there is a $j$ such
that $C_{I_j} \subseteq \mathcal{O}_1 = M/Q$.  Then the intersection
\[
\bigcap_j C_{I_j}(E_\bullet)
\]
is contained in $\mathcal{O}_1 = M/Q$, and by assumption we can make this intersection non-empty and
proper inside this orbit by shifting the Schubert varieties by $M$, proving the theorem in this case.

Now assume that none of the Schubert varieties are contained in $\mathcal{O}_1$.  Then
for every $j$, there is an $i \in I_j$ satisfying $i > r+1$.  Furthermore, there
exists a $j_0$ such that $I_{j_0}$ contains an integer $i$ satisfying $i < r$.  This 
follows from the calculation of the difference in codimension in Proposition \ref{CodimProp2}
and Proposition \ref{ExpDimProp}.  Without loss of generality, suppose $j_0 = 1$.
Consider the following intersection:
\[
C_{I_1}(E_\bullet(I_1)) \cap \bigcap_{j=2}^n C_{I_j}(E_\bullet)
\]
Then we claim that we can make the above intersection proper in each
orbit of $M$, and nonempty in $\mathcal{O}_2$, by shifting these Schubert varieties
independently by $M$.  Now by Proposition \ref{SchubOrbDim}, the intersection
is empty in $\mathcal{O}_1$ and $\mathcal{O}_2'$.  Clearly we can make the intersection
proper in the open orbit $\mathcal{O}_3$.

We claim that for generic $m_1, \ldots, m_n \in M$ the intersection
\[
X  = \mathcal{O}_2 \cap ( m_1C_{I_1}(E_\bullet(I_1)) \cap \bigcap_{j=2}^n m_jC_{I_j}(E_\bullet) )
\]
is exactly a point.  Recall the morphism $\mathcal{O}_2 \rightarrow \text{IG}(k, 2(r-1)) \times \text{Gr}(1,2)$.
As shown in the proof of Proposition \ref{SchubOrbDim}, for $j>1$, each 
$\mathcal{O}_2 \cap C_{I_j}(E_\bullet)$ is fibered over $C_{I_j}^M \times \text{Gr}(1,2)$
by open subsets of linear spaces in $\text{Gr}(k,k+1)$ of dimension $|I_j^M > r-1|$, and similarly
$\mathcal{O}_2 \cap C_{I_1}(E_\bullet(I_1))$ is fibered over $C_{I_1}^M $
by open subsets of linear spaces of dimension $|I_1^M > r-1|$.  Choosing generic $m_1, \ldots m_n$
we can make the intersection $\bigcap_j m_jC_{I_j}^M$ equal to a single point $x$ by
assumption.  Then since $M$ acts transitively on $\mathcal{O}_2$, it acts transitively
on the fiber over $x$, and therefore by shifting the fibers by generic elements of the
stabilizer of $x$ we see that the intersection $X$ is isomorphic to a non-empty open subset
of a linear space of dimension 
\begin{align*}
k - \sum_j (k - |I_j^M > r-1|) &= k - \sum_j |I_j^M \leq r-1| \\
                               &= \frac{1}{2}(2k - \sum_j \text{codim}(C_{I_j}) - \text{codim}(C_{I_j}^M)) =0,
\end{align*}
where the last two equalities follow by Propositions \ref{CodimProp2} and \ref{ExpDimProp}. Therefore, $X$ is a single point, finishing the proof.
\end{proof}

\subsection{Proof of the eigencone theorems}

In this section we use the cohomology result (Theorem \ref{CohomThm}) proven in the previous section
to prove the eigencone results stated in the introduction.  The proof also relies on the theorem of
Belkale and Kumar reducing the inequalities determining the eigencone to those corresponding
to Levi-movable products (Theorem \ref{EigIneqThm}).  Parts (3) and (5) also depend on the
work of Braley and Lee \cite{BRTHS,LEETHS}.

\begin{proof}[Proof of Theorem \ref{MainThm}]
(2)  The inclusion $M \subseteq G$  induces a map of eigencones $\econe{n}{M} \rightarrow \econe{n}{G}$.
When restricted to the first factor $M_1 \cong \text{Sp}(2(r-1))$ the map is an isometric embedding 
by the description of the fundamental weights of 
$M$ and $G$ in section \ref{PrelimSection}.  Furthermore note that the image of the
dominant chamber of $M_1$ is identified with the subcone of the dominant chamber of $G$ where the
coefficient of $\omega_r$ is zero.  Therefore by Theorem \ref{CohomThm} and Corollary
\ref{SubconeCor}, the proof is finished.

(1) This follows from the isomorphism between the eigencones of groups of types B and C (see section \ref{EigenconeSec}).
Note that $\omega_r^C = \omega_r^B$ for $i<r$ $\omega_r^C = 2\omega_r^B$.

(3)  This follows from part (1) and the main theorem in \cite{BRTHS}, which identifies the subcone of 
$\econe{n}{D_r}$ in which the coefficients of $\omega_{r-1}$ and $\omega_r$ are the same
with $\econe{n}{B_{r-1}}$.

(4)  We follow the same strategy as in part (2). In this case $M \cong \text{SL}(2) \times \text{SL}(2)$, and
$M \subseteq G$ corresponds to the sub-root-system of $G$ with positive roots $\beta = 3 \alpha_1 + 2 \alpha_2$
and $\beta' = \alpha_1$.  The factor corresponding to $\beta$ will be labeled $M_1$.  It is easy to check
that the fundamental weights are $\nu = \frac{1}{2}\omega_2$ and $\nu' = \omega_1 -  \frac{1}{2}\omega_2$,
so that $a\omega_1 + b\omega_2$ equals $(a+2b)\nu + a\nu'$. 
Let $P$ be the parabolic subgroup of $G$ which excludes $\alpha_2$, and $Q$ be the parabolic subgroup of $M$ 
which excludes $\beta$.  We need to check that non-zero cohomology products over $M/Q$ correspond to non-zero
products over $G/P$.  But there are only two Schubert cells in $M/Q$: a point and the big cell.  Therefore
it is sufficient to check that the big cell in $M/Q$ corresponds to the big cell in $G/P$ via the inclusion
of Weyl groups $W_M \subseteq W$.  The required calculations were done using the LiE software package \cite{LiE},
and $w=t \in W_M^Q$ maps to $s_2s_1s_2s_1s_2$ in $W$, which corresponds to the big cell in $G/P$.

(5)   Let $M_1 \cong \text{G}_2 \subseteq \text{F}_4$ be the subgroup as described in the theorem.
It is easy to see that given a weight $\lambda = a\omega_1 + b\omega_2$ of $\text{F}_4$, restricting
the weight to $M_1$ we get $\lambda_{|M_1} = 3b\nu_1 + a\nu_2$.  As above, it suffices to finish the
proof of the theorem to show that for any Levi-movable product $\sigma_1^M \cdots \sigma_n^M = 1\text{[pt]}$
the corresponding product in $G/P$ is non-zero.  Lee computed \cite{LEETHS} that any such Levi-movable product
is simply the product of a class $\sigma^M$ and its dual satisfying $\sigma^M\cdot(\sigma^M)^* = 1\text{[pt]}$.
Therefore it is sufficient to check that the map $W_M^Q \rightarrow W^P$ commutes with taking
duals for each parabolic $Q$ of $M_1$.  There are two maximal parabolic subgroups, $Q_1$ and $Q_2$,
corresponding to omitting the roots $\beta_1$ and $\beta_2$, respectively.  These parabolics
are the restrictions of the parabolics $P_4$ and $P_1$ of $G$, omitting $\alpha_4$ and $\alpha_1$,
respectively.  Tables \ref{G2Table}, \ref{G2F4Table}, and \ref{F4Table} list the sets of minimal
coset representatives, their duals, and the correspondence between $W_M^Q$ and $W^P$, which finishes
the proof of the theorem.  The dual computations are reproduced from \cite{LEETHS}, and the rest
was computed using the LiE system.
\end{proof}

Corollary \ref{MainCor} follows from the description of the weights in the above proof.  Finally,
we prove the projection theorem using the same strategy.

\begin{proof}[Proof of Theorem \ref{ProjThm}]
The proof is essentially identical to part (2) of the proof of Theorem \ref{MainThm}.
Again it suffices to consider the case when $G=\text{Sp}(2r)$ and $M=\text{Sp}(2(r-1))\times\text{SL}(2)$.
For $w \in W_M^Q$, we want to compare 
$\langle \omega_P, w^{-1}\lambda \rangle = \langle \omega_P, w^{-1}\sum_i a_i \omega_i \rangle$ with 
$\langle \omega_P, w^{-1}\pi(\lambda) \rangle = \langle \omega_P, w^{-1}(\sum_{i=1}^{r-2} a_i \nu_i + (a_{r-1} + a_r)\nu_{r-1}) \rangle$.  They differ
by $\langle \omega_P, w^{-1} a_r \nu_r \rangle$, and since $\nu_r = \frac{1}{2}\alpha_r$, this quantity
is zero.  Therefore given $\vec{\lambda} \in \econe{n}{G}$ and  a Levi-movable product 
$\sigma_{w_1}^M \cdots \sigma_{w_n}^M = \text{[pt]}$, as above we have
$\sum_i \langle \omega_P, w_i^{-1} \pi(\lambda_i) \rangle = \sum_i \langle \omega_P, w_i^{-1} \lambda_i \rangle \leq 0$,
finishing the proof.
\end{proof}

\renewcommand{\arraystretch}{1.25}

\begin{table}[!h]
\begin{center}
\begin{tabular}{|R|R||R|R|}
\hline
w \in W_{G_2}^{Q_1} & \text{Dual} & w \in W_{G_2}^{Q_2} & \text{Dual} \\
\hline 
e     & 12121 & e     &  21212   \\
1     & 2121  & 2     &  1212  \\
21    & 121   & 12    &  212  \\
121   & 21    & 212   &  12  \\
2121  & 1     & 1212  &  2  \\
12121 & e     & 21212 &  e  \\
\hline
\end{tabular}
\end{center}
\caption{$G_2$ Schubert cells}
\label{G2Table}
\end{table}

\begin{table}[!h]
\begin{center}
\begin{tabular}{|R|R||R|R|}
\hline
w \in W_{G_2}^{Q_1} & \text{Image in } W_{F_4}^{P_4} & w \in W_{G_2}^{Q_2} & \text{Image in } W_{F_4}^{P_1} \\
\hline 
e      & e               & e     &  e   \\
1      & 43234           & 2     &  1  \\
21     & 143234          & 12    &  2324321  \\
121    & 232143234       & 212   &  12324321  \\
2121   & 1232143234      & 1212  &  23214321324321  \\
12121  & 432132343213234 & 21212 &  123214321324321  \\
\hline
\end{tabular}
\end{center}
\caption{$G_2$ to $F_4$ Weyl data}
\label{G2F4Table}
\end{table}

\begin{table}[!h]
\begin{center}
\begin{tabular}{|R|R||R|R|}
\hline
w \in W_{F_4}^{P_4} & \text{Dual} & w \in  W_{F_4}^{P_1} & \text{Dual} \\
\hline 
e               & 432132343213234 &  e               & 123214321324321\\
43234           & 1232143234      &  1               & 23214321324321\\
143234          & 232143234       &  2324321         & 12324321 \\
232143234       & 143234          &  12324321        & 2324321\\
1232143234      & 43234           &  23214321324321  & 1 \\
432132343213234 & e               &  123214321324321 & e  \\
\hline
\end{tabular}
\end{center}
\caption{$F_4$ Schubert cells}
\label{F4Table}
\end{table}

%%%%%%%%%%%%%%  BIBLIOGRAPHY  %%%%%%%%%%%%%%
\bibliographystyle{acm}
\bibliography{ConfRed}

\begin{thebibliography}{10}

\bibitem{BEL01}
{\sc Belkale, P.}
\newblock Local systems on {$\Bbb P^1-S$} for {$S$} a finite set.
\newblock {\em Compositio Math. 129}, 1 (2001), 67--86.

\bibitem{B06}
{\sc Belkale, P.}
\newblock Geometric proofs of {H}orn and saturation conjectures.
\newblock {\em J. Algebraic Geom. 15}, 1 (2006), 133--173.

\bibitem{BKISOAXV}
{\sc Belkale, P., and Kumar, S.}
\newblock Eigencone, saturation and {H}orn problems for symplectic and odd
  orthogonal groups.
\newblock arXiv:0708.0398 [math.RT].

\bibitem{BK}
{\sc Belkale, P., and Kumar, S.}
\newblock The multiplicative eigenvalue problem and deformed quantum
  cohomology.
\newblock arXiv:1310.3191 [math.AG].

\bibitem{BK06}
{\sc Belkale, P., and Kumar, S.}
\newblock Eigenvalue problem and a new product in cohomology of flag varieties.
\newblock {\em Invent. Math. 166}, 1 (2006), 185--228.

\bibitem{BKISO}
{\sc Belkale, P., and Kumar, S.}
\newblock Eigencone, saturation and {H}orn problems for symplectic and odd
  orthogonal groups.
\newblock {\em J. Algebraic Geom. 19}, 2 (2010), 199--242.

\bibitem{BS00}
{\sc Berenstein, A., and Sjamaar, R.}
\newblock Coadjoint orbits, moment polytopes, and the {H}ilbert-{M}umford
  criterion.
\newblock {\em J. Amer. Math. Soc. 13}, 2 (2000), 433--466 (electronic).

\bibitem{BLIE46}
{\sc Bourbaki, N.}
\newblock {\em Lie groups and {L}ie algebras. {C}hapters 4--6}.
\newblock Elements of Mathematics (Berlin). Springer-Verlag, Berlin, 2002.
\newblock Translated from the 1968 French original by Andrew Pressley.

\bibitem{BRTHS}
{\sc Braley, E.}
\newblock {\em Eigencone {P}roblems for {O}dd and {E}ven {O}rthogonal
  {G}roups}.
\newblock ProQuest LLC, Ann Arbor, MI, 2012.
\newblock Thesis (Ph.D.)--The University of North Carolina at Chapel Hill.

\bibitem{BKT09}
{\sc Buch, A.~S., Kresch, A., and Tamvakis, H.}
\newblock Quantum {P}ieri rules for isotropic {G}rassmannians.
\newblock {\em Invent. Math. 178}, 2 (2009), 345--405.

\bibitem{LiE}
{\sc {Computer Algebra Group of CWI}}.
\newblock {\em LiE Software Package ({V}ersion 2.2.2)}, 2000.
\newblock \url{http://wwwmathlabo.univ-poitiers.fr/~maavl/LiE/}.

\bibitem{FULINT}
{\sc Fulton, W.}
\newblock {\em Intersection theory}, second~ed., vol.~2 of {\em Ergebnisse der
  Mathematik und ihrer Grenzgebiete. 3. Folge. A Series of Modern Surveys in
  Mathematics [Results in Mathematics and Related Areas. 3rd Series. A Series
  of Modern Surveys in Mathematics]}.
\newblock Springer-Verlag, Berlin, 1998.

\bibitem{FEIG}
{\sc Fulton, W.}
\newblock Eigenvalues, invariant factors, highest weights, and {S}chubert
  calculus.
\newblock {\em Bull. Amer. Math. Soc. (N.S.) 37}, 3 (2000), 209--249
  (electronic).

\bibitem{HRN62}
{\sc Horn, A.}
\newblock Eigenvalues of sums of {H}ermitian matrices.
\newblock {\em Pacific J. Math. 12\/} (1962), 225--241.

\bibitem{HUM}
{\sc Humphreys, J.~E.}
\newblock {\em Reflection groups and {C}oxeter groups}, vol.~29 of {\em
  Cambridge Studies in Advanced Mathematics}.
\newblock Cambridge University Press, Cambridge, 1990.

\bibitem{KLMJ09}
{\sc Kapovich, M., Leeb, B., and Millson, J.~J.}
\newblock Polygons in buildings and their refined side lengths.
\newblock {\em Geom. Funct. Anal. 19}, 4 (2009), 1081--1100.

\bibitem{KLY98}
{\sc Klyachko, A.~A.}
\newblock Stable bundles, representation theory and {H}ermitian operators.
\newblock {\em Selecta Math. (N.S.) 4}, 3 (1998), 419--445.

\bibitem{KNTAO}
{\sc Knutson, A., and Tao, T.}
\newblock The honeycomb model of {${\rm GL}_n({\bf C})$} tensor products. {I}.
  {P}roof of the saturation conjecture.
\newblock {\em J. Amer. Math. Soc. 12}, 4 (1999), 1055--1090.

\bibitem{KSKAC}
{\sc Kumar, S.}
\newblock {\em Kac-{M}oody groups, their flag varieties and representation
  theory}, vol.~204 of {\em Progress in Mathematics}.
\newblock Birkh\"auser Boston, Inc., Boston, MA, 2002.

\bibitem{AEPS}
{\sc Kumar, S.}
\newblock A survey of the additive eigenvalue problem.
\newblock {\em Transform. Groups 19}, 4 (2014), 1051--1148.
\newblock With an appendix by M. Kapovich.

\bibitem{LEETHS}
{\sc Lee, B.}
\newblock {\em A {C}omparison of {E}igencones {U}nder {C}ertain {D}iagram
  {A}utomorphisms}.
\newblock ProQuest LLC, Ann Arbor, MI, 2012.
\newblock Thesis (Ph.D.)--The University of North Carolina at Chapel Hill.

\bibitem{RSS10}
{\sc Ressayre, N.}
\newblock Geometric invariant theory and the generalized eigenvalue problem.
\newblock {\em Invent. Math. 180}, 2 (2010), 389--441.

\bibitem{WEYL}
{\sc Weyl, H.}
\newblock Das asymptotische {V}erteilungsgesetz der {E}igenwerte linearer
  partieller {D}ifferentialgleichungen (mit einer {A}nwendung auf die {T}heorie
  der {H}ohlraumstrahlung).
\newblock {\em Math. Ann. 71}, 4 (1912), 441--479.

\end{thebibliography}

\end{document}